\documentclass[12pt, reqno]{amsart}
\usepackage{amsmath,amsfonts,amsthm,amssymb,latexsym}
\usepackage[a4paper,margin=3cm]{geometry}
\usepackage{graphicx}
\usepackage{url}
\usepackage[latin1]{inputenc}
\usepackage[T1]{fontenc} 
\usepackage{graphicx}
\usepackage[english]{babel}
\usepackage{calc}
\usepackage{amsmath,amsthm}
\usepackage{url}
\usepackage{times, amssymb, amscd, mathrsfs, graphicx, color}  
\usepackage{enumerate}
\usepackage{hyperref}
\usepackage{dsfont}
 \usepackage[all]{xy}
\definecolor{darkred}{rgb}{0.9,0.1,0.1}

\newtheorem{theo}{Theorem}[section]

\newtheorem{coro}[theo]{Corollary}
\newtheorem{lem}[theo]{Lemma}
\newtheorem{exe}[theo]{Example}

\newtheorem{Rq}[theo]{Remark}

\theoremstyle{plain}
\newtheorem*{hypo*}{Assumption}

\newcommand{\E}{\mathbb{E}}                                            

\newcommand{\var}{\textrm{Var}}

\title{Birth and Death Process in Mean Field type interaction} %

\date{\today}

\author{Marie-No\'{e}mie~\textsc{Thai}} %
\address[Marie-No\'{e}mie~\textsc{Thai}]{CEREMADE, Universit\'e Paris-Dauphine,
  France} %
\email{\url{noemiethai@yahoo.fr}}

\begin{document}

\maketitle
\begin{abstract} 
The aim of this paper is to study the asymptotic behavior of a system of birth and death processes in mean field type interaction in discrete space. We first establish the exponential convergence of the particle system to equilibrium for a suitable Wasserstein coupling distance. The approach provides an explicit quantitative estimate on the rate of convergence. We prove next a uniform propagation of chaos property. As a consequence, we show that the limit of the associated empirical distribution, which is the solution of a nonlinear differential equation, converges exponentially fast to equilibrium. This paper can be seen as a discrete version of the particle approximation of the McKean-Vlasov equations and is inspired from previous works of Malrieu and al and Caputo, Dai Pra and Posta. 
\end{abstract}

{\footnotesize \textbf{AMS 2000 Mathematical Subject Classification:} 60K35, 60K25, 60J27, 60B10, 37A25.}


\renewcommand*\abstractname{ \ }
\begin{abstract}
\tableofcontents
\end{abstract}

\section{Introduction}
The concept of \textit{mean field interaction} arised in statistical physics with Kac \cite{Kac56} and then McKean \cite{McKean67} in order to describe the collisions between particles in a gas, and has later been applied in other areas such as biology or communication networks. A particle system is in mean field interaction when the system acts over one fixed particle through the empirical measure of the system. For continuous interacting diffusions, this linear particle system has been introduced in order to approximate the solution of a non linear equation, the so-called McKean-Vlasov equation, and has been extensively studied by many authors. But, to our knowledge, there are few results in discrete space. 

In this paper, we give a discrete version of the particle approximation of the McKean-Vlasov equations. We consider a system of $N$ particles $X^{1,N}, \dots, X^{N,N}$ evolving in $\mathbb{N}=\{0,1,2,\dots\}$, each one according to a birth and death process in mean field type interaction. Namely, a single particle evolves with a rate which depends on both its own position and the mean position of all particles. At time $0$, the particles are independent and identically distributed and at time $t>0$, the interaction is given in terms of the mean of the particles at this time.
Let $q^{+}: \mathbb{N} \times \mathbb{R}_{+} \rightarrow \mathbb{R}_{+}$ and $q^{-}: \mathbb{N} \times \mathbb{R}_{+} \rightarrow \mathbb{R}_{+}$ be the interaction functions. For any particle $k \in \{1, \dots, N\}$ and position $i \in \mathbb{N}$, the transition rates at time $t$, given $\left(X_t^{1,N}, \dots, X_t^{N,N}\right)$ and $X_t^{k,N}=i$ are
\[ 
  \begin{array}{ c c c c c}
     i & \rightarrow & i+1 & \text{with rate} \ b_i + q^{+}(i,M_{t}^{N}), &\\
     i & \rightarrow & i-1 & \text{with rate} \ d_i + q^{-}(i,M_{t}^{N}), & \text{for} \ i\geq 1\\
     i & \rightarrow & j & \text{with rate} \ 0, & \ \text{if} \ j\notin \{i-1,i+1\},\\
  \end{array} 
\]

where $b_i>0$ for $i \geq 0$, $d_i >0$ for $i\geq 1$, $d_0 = q^{-}(0,m) = 0$ for all $m\in \mathbb{R}_+$ and $M_{t}^{N}$ denotes the mean of the $N$ particles at time $t$, defined by
$$
M_{t}^{N}= \dfrac{1}{N} \sum_{i=1}^{N} X_{t}^{i,N}.
$$

By the positivity of the rates $b_i,d_i$, the process is irreducible but not necessary reversible.
The generator of the particle system acts on bounded functions $f: \mathbb{N}^{N} \rightarrow \mathbb{R}$ as follows
\begin{align}
\label{eq:generator}
\mathcal{L}f(x) 
&= \sum_{i=1}^{N} \Big[\left(b_{x_{i}} + q^{+}(x_{i},M^{N})\right) \left(f(x+e_i)-f(x)\right)\\
& \qquad \qquad  +\left(d_{x_{i}} + q^{-}(x_{i},M^{N})\right) \left(f(x-e_i)-f(x)\right) \mathds{1}_{x_i>0} \Big]\nonumber,
\end{align}
where $e_i, i\in \mathbb{N}$ is the canonical vector and
$$
M^{N}=M^{N}(x)= \dfrac{1}{N} \sum_{i=1}^{N} x_i.
$$

\bigskip 

For such systems, we consider the limiting behavior as time and the size of the system go to infinity. As $N$ tends to infinity, this leads to the propagation of chaos phenomenon: the law of a fixed number of particles becomes asymptotically independent as the size of the system goes to infinity \cite{Sznitman89}. Sznitman \cite[Proposition 2.2]{Sznitman89} (see also M\'{e}l\'{e}ard \cite[Proposition 4.2]{Meleard96}) showed that, in the case of exchangeable particles, this is equivalent to the convergence in law of the empirical measure to a deterministic measure.
This limiting measure, denoted by $u$ and often called the mean field limit, is characterized by being the unique weak solution of the nonlinear master equation
\begin{equation}
\label{eq:nonlinearMasterEq}
\left\{
    \begin{array}{ll}
       \dfrac{d}{dt} \langle u_t,f\rangle = \langle u_t, \mathcal{G}_{u_t}f\rangle \\
        \quad \quad u_0 \in \mathcal{M}_1(\mathbb{N}),
    \end{array}
\right.
\end{equation}

where $\mathcal{M}_{1}(\mathbb{N})$ is the set of probability measures on $\mathbb{N}$ and $\mathcal{G}_{(\cdot)}$ is the operator defined through the behavior of the particle system by 
\begin{equation}
\label{eq:generatorParticle}
\mathcal{G}_{u_t}f(i)= \left(b_i + q^{+}(i,\Vert u_t \Vert)\right)\left(f(i+1)-f(i)\right) + \left(d_i + q^{-}(i,\Vert u_t \Vert)\right)\left(f(i-1)-f(i)\right) \mathds{1}_{i>0},
\end{equation}
for every $i \in \mathbb{N}$. For any probability measure $u$ and bounded function $f$, the linear form $\langle u,f \rangle$ is defined by 
$$
\langle u,f \rangle  = \sum_{k \in \mathbb{N}} f(k) u(\{k\}),
$$
and 
$$
\Vert u \Vert  = \langle u(dx),x \rangle \ \text{is the first moment of $u$}.
$$

\bigskip

We use the notation $\Vert \cdot \Vert$ to be consistent with the previous works \cite{DTZ05,FZ92}.
The propagation of chaos phenomenon can be seen as giving both the asymptotic behavior of an interacting particle system and an approximation of solutions of nonlinear differential equations. 

\bigskip

We introduce a stochastic process $(\overline{X}_t)_{t\geq 0}$ whose time marginals are solutions of the nonlinear equation  \eqref{eq:nonlinearMasterEq}. This process is defined as a solution of the following martingale problem: for a nice test function $\varphi$
$$
\varphi(\overline{X}_t)-\varphi(\overline{X}_0)-\int_0^{t} \mathcal{G}_{u_s}\varphi(\overline{X}_s) ds \ \text{is a martingale},
$$
where $u_s = u \circ \overline{X}_s^{-1}$ and $u_0 = u\circ \overline{X}_0^{-1}$.

Under some assumptions, the existence and uniqueness of the solution of the martingale problem are insured. If $q^{-}\equiv 0$, this has been proved by Dawson, Tang and Zhao in \cite{DTZ05} and before by  Feng and Zheng \cite{FZ92} in the case where $q^{+}(i,\Vert u_t \Vert)=\Vert u_t \Vert$ for all $i\in \mathbb{N}$. This unique solution is a solution of \eqref{eq:nonlinearMasterEq} (cf  \cite[Theorem 2]{DTZ05}, \cite[Theorem 1.2]{FZ92}). 

\bigskip

For any $t\geq 0$, let us denote by $\mu_{t}^{N}$ the empirical distribution of $(X^{1,N}, \dots, X^{N,N})$ at time $t$ defined by
\begin{equation}
\label{eq:meaEmpirique}
\mu_{t}^{N} = \dfrac{1}{N} \sum_{i=1}^{N} \delta_{X_{t}^{i,N}} \in \mathcal{M}_{1}(\mathbb{N}).
\end{equation}

This measure is a random measure on $\mathbb{N}$ and we note that the first moment of $\mu_{t}^{N}$ is exactly the value of $M_t^{N}$. Our goal is to quantify the following limits (if they exist)
\begin{figure}[hbtp]
$$
\xymatrix{
    & \mu_t^{N} \ar[rd]^{t \rightarrow \infty} \ar[ld]_{N \rightarrow \infty} & \\
	u_t \ar[rd]_{t \rightarrow \infty} &  & \mu_{\infty}^{N} \ar[ld]^{N \rightarrow \infty} \\
	& u_{\infty} &
}
$$  
\end{figure}
\bigskip

Similar problems for interacting diffusions of the form
$$
d X_t^{i,N} = \sqrt{2} dB_t^{i,N} - \nabla V(X_t^{i,N})dt - \dfrac{1}{N} \sum_{j=1}^{N} \nabla W(X_t^{i,N}-X_t^{j,N})dt \quad 1\leq i \leq N, 
$$
have been studied in several works \cite{CGM10,Malrieu01,Malrieu03}. Here, $(B_t^{i,N})_{1\leq i \leq N}$ are $N$ independent Brownian motions, $V$ and $W$ are two potentials and the symbol $\nabla$ stands for the gradient operator.
The mean field limit associated to these dynamics satisfies the so-called McKean-Vlasov equation.

 
Our initial motivation was the study of the Fleming-Viot type particle systems. In this system, the particles evolve as independent copies of a Markov process until one of them reaches the absorbing state. At this moment, the absorbed particle goes instantaneously to a state chosen with the empirical distribution of the particles remaining in the state space. For example, if we consider $N$ random walks on $\mathbb{N}$ with a drift towards the origin (cf \cite{AT12,Maric14}), the Fleming-Viot system can be interpreted as a system of $N$ $M/M/1$ queues in interaction: when a queue is empty, another duplicates. The Fleming-Viot process has been introduced in order to approximate the solutions of a nonlinear equation: the quasi-stationary distributions. For results and related methods, we refer to \cite{AFG,AT12,BHIM,CT13,FM07,AFGJ} and references therein.

\bigskip
\subsection*{Long time behavior of the particle system}
Our starting point is the long time behavior of the interacting particle system.
We show that under some conditions, the particle system converges exponentially fast to equilibrium for a suitable Wasserstein coupling distance. Let us describe first the different distances that we use and the assumptions that we make. For $x,\overline{x} \in \mathbb{N}^{N}$, let $d$ be the $l^{1}$-distance defined by
\begin{equation}
\label{eq: distance}
d(x,\overline{x}) = \sum_{k=1}^{N} \vert x_{k}- \overline{x}_{k} \vert,
\end{equation}
and for any two probability measures $\mu$ and $\mu'$ on $\mathbb{N}$, let $\mathcal{W}(\mu, \mu')$ be the Wasserstein coupling distance between these two laws defined by
\begin{equation}
\label{eq:Wassdistance}
\mathcal{W}(\mu, \mu') = \inf_{\substack{X \sim \mu\\ \overline{X} \sim \mu'}} \E\left[ d(X,\overline{X})\right],
\end{equation}
where the infimum runs over all the couples of random variables with marginal laws $\mu$ and $\mu'$. We note that, as the distance $d$ is the $l^{1}$-distance, the associated Wasserstein distance $\mathcal{W}$ defined in \eqref{eq:Wassdistance} is in fact the so-called $\mathcal{W}_1$-Wasserstein distance. Let us assume that:
\begin{hypo*} \
\begin{enumerate}
\item[$(A)$](\textit{Convexity condition}) There exists $\lambda>0$ such that 
$$
\nabla^{+}(d - b) \geq \lambda, 
$$
where for every $n\geq 0$ and $f: \mathbb{N}\rightarrow \mathbb{R}_+$
$$
\nabla^{+}(f)(n) = f(n+1)-f(n).
$$

\item[$(B)$](\textit{Lipschitz condition}) The function $q^{+}$ (resp. $q^{-}$) is 
non-decreasing (resp. non-increasing) in its second component and there exists $\alpha>0$ such that for any $\left( k_1,l_1\right),\left(k_2,l_2\right) \in \mathbb{N} \times \mathbb{R}_+$
$$
\vert \left(q^{+}-q^{-}\right)(k_1,l_1)- \left(q^{+}-q^{-}\right)(k_2,l_2) \vert \leq \alpha \left(\vert k_1-k_2 \vert + \vert l_1-l_2\vert \right).
$$
\end{enumerate}
\end{hypo*}

\bigskip
An example of rates satisfying the assumptions $(A)$ and $(B)$ is given in Example \ref{ex:example}.
Denoting by $\mathrm{Law}(Y)$ the law of a random variable $Y$, we have

\begin{theo}[Long time behavior]
\label{th:coupling}
Assume that Assumptions $(A)$ and $(B)$ are satisfied. Then, for all $t\geq 0$
\begin{equation}
\label{eq:coupling}
\mathcal{W}(\mathrm{Law}(X_t^{N}), \mathrm{Law}(Y_t^{N})) \leq e^{-(\lambda-2\alpha) t} \mathcal{W}(\mathrm{Law}(X_0^{N}), \mathrm{Law}(Y_0^{N})),
\end{equation} 
where $(X_t^{N})_{t\geq 0}$ and $(Y_t^{N})_{t\geq 0}$ are two processes generated by \eqref{eq:generator}.\\ 
In particular, if $\lambda-2\alpha>0$ there exists a unique invariant distribution $\lambda_{N}$ satisfying for every $t \geq 0$,
$$
\mathcal{W}(\mathrm{Law}(X_t^{N}), \lambda_{N}) \leq e^{-(\lambda-2\alpha) t} \mathcal{W} (\mathrm{Law}(X_0^{N}), \lambda_{N}).
$$
\end{theo}

To our knowledge, it is the first theorem which establishes an exponential convergence of this interacting particle system in discrete state space and with an explicit rate. For continuous interacting diffusions, Malrieu \cite{Malrieu01, Malrieu03} proved that under some assumptions of convexity of the potentials and based on the Bakry-\'{E}mery criterion, the system of particles associated with the McKean-Vlasov equations converges exponentially fast to equilibrium with a rate that does not depend on $N$ in terms of relative entropy.

\bigskip

\subsection*{Propagation of chaos} When $N$ is large, we would like to show that the random empirical distribution $\mu_t^{N}$ is close to the deterministic measure $u_t$, solution of \eqref{eq:nonlinearMasterEq} at time $t$. For this convergence, one of the interesting points is to obtain a quantitative bound.

This behavior has been studied by Dawson,Tang and Zhao \cite{DTZ05} and Dawson and Zheng \cite{DZ91}, the key ingredient of the proof being the tightness of $\{\mu^{N}: N\geq 1\}$. For interacting diffusions, this has been studied particularly by Sznitman \cite{Sznitman89}, M\'{e}l\'{e}ard \cite{Meleard96} or Malrieu \cite{Malrieu01, Malrieu03}. The central limit theorem and large deviation principles for such models have also been studied by many authors \cite{DG88,Feng94,HT81,McKean67,ST85,Sznitman84}. \\

In the following theorem, we prove the propagation of chaos property and show that this property is uniform in time. For this purpose, we introduce a family of independent processes $(\overline{X}^{i})_{i \in \{1, \dots, N\}}$ of law $u_t$ at each time $t$ such that for $i\in \{1, \dots, N\}$
\begin{itemize}
\item $\overline{X}_0^{i} = X_0^{i,N}$
\item the transition rates of $\overline{X}^{i}$ at time $t$ are given by
\[ 
  \begin{array}{ c c c c c}
     i & \rightarrow & i+1 & \text{with rate} \ b_i + q^{+}(i,\Vert u_t\Vert), &\\
     i & \rightarrow & i-1 & \text{with rate} \ d_i + q^{-}(i,\Vert u_t\Vert), & \text{for} \ i\geq 1\\
     i & \rightarrow & j & \text{with rate} \ 0, & \ \text{if} \ j\notin \{i-1,i+1\}.\\
  \end{array} 
\]
\end{itemize}
For $i\in \{1, \dots, N\}$ and $t\geq 0$, the process $\overline{X}_t^{i}$ is said to be nonlinear in the sense that its dynamic depends on its law. 
 
\begin{theo}[Uniform Propagation of chaos]
\label{th:PropChaos}
Assume that Assumptions $(A)$ and $(B)$ are satisfied. Assume moreover that the assumptions of Lemma \ref{lem:expoMoment} hold. Then, if $\lambda-2\alpha>0$, there exists a coupling and a constant $K>0$ such that 
\begin{equation}
\label{eq:chaos}
\sup_{t\geq 0} \E \vert X_t^{1,N} - \overline{X}_t^{1} \vert \leq \dfrac{K}{\sqrt{N}}.
\end{equation}
\end{theo}

For continuous interacting diffusions, Malrieu \cite{Malrieu01,Malrieu03} establishes for the first time a uniform (in time) propagation of chaos in the case of uniform convexity of the potentials. Cattiaux, Guillin, Malrieu \cite{CGM10} extend his results when the potentials are no more uniformly convex. The proof is based on a coupling argument and It\^{o}'s formula. 

\bigskip

A consequence of the uniform propagation of chaos phenomenon is the convergence of the empirical measure $\mu_t^{N}$ to the solution $u_t$ of the nonlinear equation. This convergence is well known but, here, we give an explicit rate. To express this convergence, we set for a function $\varphi: \mathbb{N}\rightarrow \mathbb{R}$ 
$$
\mu_t^{N}(\varphi) = \sum_{k \in \mathbb{N}} \varphi(k) \mu_t^{N}(k), \ \text{and} \ u_t(\varphi) = \sum_{k \in \mathbb{N}} \varphi(k) u_t(k).
$$
Moreover, we denote by $\Vert \varphi \Vert_{Lip}$ the quantity defined by 
$$
\Vert \varphi \Vert_{Lip} = \sup_{\substack{ x,y\in \mathbb{N} \\ x\neq y}} \dfrac{\vert \varphi(x)-\varphi(y) \vert}{\vert x-y \vert}.
$$

\bigskip

\begin{coro}[Convergence of the empirical measure]
\label{cor:ConvEmpirique}
Under the assumptions of Theorem \ref{th:PropChaos}, and if $\lambda-2\alpha>0$, there exists a constant $C>0$ such that
\begin{equation}
\label{eq:ConvEmpirique}
\sup_{t\geq 0} \sup_{\Vert \varphi \Vert_{Lip} \leq 1} \E \vert \mu_t^{N}(\varphi) -u_t(\varphi) \vert \leq \dfrac{C}{\sqrt{N}}.
\end{equation}
\end{coro}
In particular, from Markov's inequality, we have for $\epsilon >0$ 
\begin{align}
\label{eq:DeVinequality}
\sup_{t\geq 0} \sup_{\Vert \varphi \Vert_{Lip} \leq 1}\mathbb{P}\left( \Big \vert \dfrac{1}{N} \sum_{i=1}^{N} \varphi(X_t^{i,N}) - \int \varphi du_t \Big \vert >\epsilon \right) \leq \dfrac{C}{\epsilon \sqrt{N}}.
\end{align}

\bigskip

One can expect to improve this bound and obtain an exponential one. An exponential bound has been obtained by Malrieu \cite{Malrieu01} for interacting diffusions: as said previously, under some assumptions on the convexity of potentials, Malrieu showed that the law of $(X_t^{N})_{t\geq 0}$ satisfies a logarithmic Sobolev inequality with a constant independent of $t$ and $N$. As a consequence, via the Herbst's argument, the law of the system satisfies a Gaussian concentration inequality around its mean. Our particle system does not verify a Sobolev inequality, this is the difference with interacting diffusions and that is the difficulty. 
Under an additional assumption on the number of particles $N$, Bollet, Guillin and Villani \cite{BGV07} improve the deviation inequality of Malrieu and obtain an exponential bound of $\displaystyle \mathbb{P}(\sup_{0\leq t\leq T} \mathcal{W}(\mu_t^{N}, u_t)>\epsilon)$ for every $\epsilon>0$ and $T\geq 0$.
A Poisson type deviation bound has been established by Joulin \cite{Joulin09} for the empirical measure of birth and death processes with unbounded generator. With the distance that we introduced (namely the $l_1$-distance), we can not apply directly his results. Indeed, one of the hypotheses is the existence of a constant $V$ such that
$$
\Vert \sum_y d^{2}(\cdot,y)Q(\cdot,y) \Vert_{\infty} \leq V^{2},
$$ 
where $Q$ is the transition rates matrix and $d$ a metric on $\mathbb{N}$. However, if we consider the $l_1$-distance, $V$ is infinite. So, to apply Joulin's results, we need to choose another distance in such a way that the previous assumption is satisfied.
\bigskip

Although we are not able to provide an exponential bound of \eqref{eq:DeVinequality}, we can measure how the empirical measure $\mu_t^{N}$ is close to the law $\mu_t$ in a stronger way.

\begin{coro}[Deviation inequality]
\label{cor:devInequality}
Under the assumptions of Theorem \ref{th:PropChaos}, and if $\lambda-2\alpha>0$, there exists a constant $\overline{C}>0$ such that
$$
\sup_{t\geq 0} \mathbb{P}\left( \mathcal{W}(\mu_t^{N}, \mu_t) >\epsilon \right) \leq \dfrac{\overline{C}}{\epsilon \sqrt{N}}.
$$
\end{coro}

\bigskip

\subsection*{Long time behavior of the nonlinear process}
The long time behavior of $u_t$ is a consequence of Theorems \ref{th:coupling} and \ref{th:PropChaos}. We express the convergence of $u_t$ to equilibrium with an explicit rate, under the Wasserstein distance $\mathcal{W}$.
\begin{theo}[Long time behavior]
\label{th:LTBnonlinear}
Let us assume that the assumptions of Theorem \ref{th:PropChaos} hold. Let $(u_t)_{t\geq 0}$ and $(v_t)_{t\geq 0}$ be the solutions of \eqref{eq:nonlinearMasterEq} with initial conditions $u_0$ and $v_0$ respectively. Then, under the assumption $\lambda-2\alpha>0$
\begin{equation}
\label{eq:CauchyW1}
\mathcal{W}(u_t, v_t) \leq e^{-(\lambda- 2\alpha) t} \mathcal{W}(u_0, v_0).
\end{equation}
In particular, the nonlinear process $(\overline{X}_t)_{t\geq 0}$ associated with the equation \eqref{eq:nonlinearMasterEq} has a unique invariant measure $u_{\infty}$ and
\begin{equation}
\label{eq:CvgSolNonLinear}
\mathcal{W}(u_t, u_{\infty}) \leq e^{-(\lambda-2\alpha)t}\mathcal{W}(u_0, u_{\infty}).
\end{equation}
\end{theo}

\bigskip

The exponential convergence to equilibrium of the nonlinear process has been obtained by Malrieu \cite{Malrieu01, Malrieu03} for interacting diffusions under the Wasserstein distance $\mathcal{W}_2$. To prove this convergence, he used the uniform propagation of chaos, the exponential convergence to equilibrium of the particle system and the Talagrand transport inequality $T_2$ (connecting the Wasserstein distance and the relative entropy). Later, Cattiaux, Guillin and Malrieu \cite{CGM10} complete his result by giving the distance between two solutions of the McKean-Vlasov equation (granular media equation) starting at different points. 

\bigskip

Finally, from Corollary \ref{cor:ConvEmpirique} and Theorem \ref{th:LTBnonlinear}, we have the convergence of the empirical measure under the invariant distribution that we denote by $\mu_{\infty}^{N}$.

\begin{coro}[Convergence under the invariant distribution] Assume that the assumptions of Corollary \ref{cor:ConvEmpirique} are satisfied. Assume moreover that $\lambda-2\alpha>0$. Then 
$$
\sup_{\Vert \varphi \Vert_{\infty} \leq 1} \E[\vert \mu_{\infty}^{N}(\varphi) - u_{\infty}(\varphi) \vert] \leq \dfrac{C}{\sqrt{N}}.
$$
\end{coro}

\bigskip

The remainder of the paper is as follows. Section \ref{sect:proofThmCoupling} gives the proof of Theorem \ref{th:coupling}, Section \ref{sect:proofProgChaos} the proofs of the propagation of chaos phenomenon (Theorem \ref{th:PropChaos}) and Corollaries \ref{cor:ConvEmpirique} and \ref{cor:devInequality}. Finally, Section \ref{sect:proofNonlinear} is devoted to the proof of Theorem \ref{th:LTBnonlinear}. We conclude the paper with Section \ref{sect:annexe}, where we state the non-explosion and the positive recurrence of the interacting particle system.

\section{Proof of Theorem \ref{th:coupling}}
\label{sect:proofThmCoupling}
First of all, there exists a Lyapunov function for the process $(X_t^{N})_{t\geq 0}$ which ensures its non-explosion and positive recurrence (see Appendix). Combining with
the irreducibility of the process $(X_t^{N})_{t\geq 0}$, the Foster-Lyapunov criteria ensures its ergodicity and even its exponential ergodicity, see \cite{MT93}. 
Before giving the demonstration of Theorem \ref{th:coupling} we give an example where Assumptions $(A)$ and $(B)$ are satistied.

\begin{exe}[Transition rates example]
\label{ex:example}
Let $p<q$ two positive constants and $a\geq 1$. For $k\in \mathbb{N}$ and $l\in \mathbb{R}_+$ let 
$$
b_k=pk^{a},\ d_k=qk^{a},\ q^{+}(k,l)=(l-k)_+ \  \text{and} \ q^{-}(k,l)=(k-l)_+,
$$
where for $x\in \mathbb{R}$, $x_+ = \max(x,0)$. The interaction functions $q^{+}$ and $q^{-}$ mean that the more the particles are far from their mean, the more they tend to come closer to it. Then, the transition rates satisfy the assumptions $(A)$ and $(B)$ with $\lambda = q-p$ and $\alpha=2$. We have equality in assumption $(A)$ for the $M/M/\infty$ queue $b_k=p$ and $d_k=qk$ or for the linear case $b_k=pk$ and $d_k=qk$ for $k\in \mathbb{N}$.
\end{exe}

\begin{Rq}[Caputo-Dai Pra-Posta results]
If we consider Assumption $(A)$ and the following assumptions $\nabla^{+}b \leq 0$ and $\nabla^{+}d\geq 0$, Caputo, Dai Pra and Posta \cite{CDPP} obtain estimates for the rate of exponential convergence to equilibrium of the birth and death process without interaction ($q^{+}= q^{-} \equiv 0$), in the relative entropy sense. Moreover, this exponential decay of relative entropy is convex in time. To do this, the authors control the second derivative of the relative entropy and show that its convexity leads to a modified logarithmic Sobolev inequality. Nevertheless, one of the key points of their approach is a condition of reversibility. In our case, the reversibility is not assumed and their results do not hold.
\end{Rq}


\begin{proof}[Proof of Theorem \ref{th:coupling}]
We build a coupling between two particle systems generated by \eqref{eq:generator}, $X=(X^{1}, \dots, X^{N}) \in \mathbb{N}^{N}$ and $Y=(Y^{1}, \dots, Y^{N})\in \mathbb{N}^{N}$, starting respectively from some random configurations $X_0, Y_0 \in \mathbb{N}^{N}$. The coupling that we introduce is the same as \cite{DTZ05} or \cite{FZ92}. Let $\mathbb{L} = \mathbb{L}_{1} + \mathbb{L}_{2}$ be the generator of the coupling defined by
\begin{align*}
\mathbb{L}_{1} f(X,Y)
&= \sum_{i=1}^{N} \left(b_{X^{i}} \wedge b_{Y^{i}}\right) \left( f(X+e_i, Y+e_i) -f(X,Y)\right)\\
&+ \sum_{i=1}^{N} \left(d_{X^{i}} \wedge d_{Y^{i}}\right)\left( f(X-e_i, Y-e_i) -f(X,Y)\right)\\
&+ \sum_{i=1}^{N} \left(b_{X^{i}} - b_{Y^{i}}\right)_{+} \left( f(X+e_i, Y) -f(X,Y)\right)\\
&+ \sum_{i=1}^{N} \left(b_{Y^{i}} - b_{X^{i}}\right)_{+} \left( f(X, Y+e_i) -f(X,Y)\right)\\
&+ \sum_{i=1}^{N} \left(d_{Y^{i}} - d_{X^{i}}\right)_{+}\left( f(X, Y-e_i) -f(X,Y)\right)\\
&+ \sum_{i=1}^{N} \left(d_{X^{i}} - d_{Y^{i}}\right)_{+}\left(f(X-e_i,Y) -f(X,Y)\right)\\
\end{align*}

and

\begin{align*}
\mathbb{L}_{2} f(X,Y)
&= \sum_{i=1}^{N} \left(q^{+}(X^{i}, M^{N,1}) \wedge q^{+}(Y^{i}, M^{N,2})\right) \left( f(X+e_i, Y+e_i) -f(X,Y)\right)\\
&+ \sum_{i=1}^{N} \left(q^{+}(Y^{i}, M^{N,2}) - q^{+}(X^{i}, M^{N,1})\right)_{+}\left( f(X, Y+e_i) -f(X,Y)\right)\\
&+ \sum_{i=1}^{N} \left(q^{+}(X^{i}, M^{N,1}) - q^{+}(Y^{i}, M^{N,2})\right)_{+}\left( f(X+e_i, Y) -f(X,Y)\right)\\
&+ \sum_{i=1}^{N} \left(q^{-}(X^{i}, M^{N,1}) \wedge q^{-}(Y^{i}, M^{N,2})\right) \left( f(X-e_i, Y-e_i) -f(X,Y)\right)\\
&+ \sum_{i=1}^{N} \left(q^{-}(Y^{i}, M^{N,2}) - q^{-}(X^{i}, M^{N,1})\right)_{+}\left( f(X, Y-e_i) -f(X,Y)\right)\\
&+ \sum_{i=1}^{N} \left(q^{-}(X^{i}, M^{N,1}) - q^{-}(Y^{i}, M^{N,2})\right)_{+}\left( f(X-e_i, Y) -f(X,Y)\right),\\
\end{align*}
where $M^{N,1}$ (resp. $M^{N,2}$) represents the mean of the particle system $X$ (resp. $Y$).
We can easily verify that if a measurable function $f$ on $\mathbb{N}^{N}\times \mathbb{N}^{N}$ does not depend on its second (resp. first) variable; that is, with a slight abuse of notation:
$$
\forall X,Y\in \mathbb{N}^{N}, \ f(X,Y)=f(X) \ (\text{resp. } \ f(X,Y)=f(Y)),
$$
then $\mathbb{L} f(X,Y)= \mathcal{L} f(X)$ (resp. $\mathbb{L} f(X,Y)= \mathcal{L} f(Y)$), where $\mathcal{L}$ is defined in \eqref{eq:generator}. This property ensures that the couple $(X_t,Y_t)_{t\geq 0}$ generated by $\mathbb{L}$  is a well-defined coupling of processes generated by $\mathcal{L}$. 
Applying the generator $\mathbb{L}$ to the distance $d$ defined in \eqref{eq: distance}, we obtain, on the one hand
$$
\mathbb{L}_{1} d(X,Y) = \sum_{i=1}^{N} K_i,
$$
where for all $i \in \{1, \dots, N\}$
\begin{align*}
K_i
&= \left(b_{X^{i}} - b_{Y^{i}}\right)_{+} \left( \vert X^{i} +1 - Y^{i} \vert - \vert X^{i} - Y^{i} \vert\right)\\
&+ \left(b_{Y^{i}} - b_{X^{i}}\right)_{+} \left( \vert X^{i} - Y^{i} -1 \vert - \vert X^{i} - Y^{i} \vert\right)\\
&+ \left(d_{Y^{i}} - d_{X^{i}}\right)_{+}\left( \vert X^{i} - Y^{i} +1 \vert - \vert X^{i} - Y^{i} \vert\right)\\
&+ \left(d_{X^{i}} - d_{Y^{i}}\right)_{+}\left( \vert X^{i} -1 - Y^{i} \vert - \vert X^{i} - Y^{i} \vert\right).\\
\end{align*}

Under Assumption (A) and using the fact that for all $x,y\geq 0$, $(x-y)_{+} - (y-x)_{+} = x-y$, there exists $\lambda>0$ such that for $i \in \{1, \dots, N\}$

\begin{align*}
K_i
&= \left(b_{X^{i}} - b_{Y^{i}} + d_{Y^{i}} - d_{X^{i}}\right)\mathds{1}_{X^{i} > Y^{i}}+ \left(b_{Y^{i}} - b_{X^{i}} + d_{X^{i}} - d_{Y^{i}}\right)\mathds{1}_{X^{i} < Y^{i}}\\
&\leq -\lambda \left( (X^{i} - Y^{i} ) \mathds{1}_{X^{i} > Y^{i}} + (Y^{i} - X^{i}) \mathds{1}_{X^{i} <Y^{i}} \right)\\
&=-\lambda  \vert X^{i} - Y^{i} \vert.  
\end{align*}

Thus,
\begin{equation*}
\label{eq: genCouplingBD}
\mathbb{L}_{1} d(X,Y) \leq -\lambda  d(X,Y).
\end{equation*}

On the other hand,
$$
\mathbb{L}_{2} d(X,Y) = \sum_{i=1}^{N} H_i,
$$
where for all $i \in \{1, \dots, N\}$
\begin{align*}
H_i
&=  \left(q^{+}(X^{i}, M^{N,1}) - q^{+}(Y^{i}, M^{N,2})\right)_{+}\left( \vert X^{i} +1 - Y^{i} \vert - \vert X^{i} - Y^{i} \vert\right)\\
&+ \left(q^{+}(Y^{i}, M^{N,2}) - q^{+}(X^{i}, M^{N,1})\right)_{+}\left( \vert X^{i} - Y^{i} -1 \vert - \vert X^{i} - Y^{i} \vert\right)\\
&+  \left(q^{-}(X^{i}, M^{N,1}) - q^{-}(Y^{i}, M^{N,2})\right)_{+}\left( \vert X^{i} - Y^{i} -1 \vert - \vert X^{i} - Y^{i} \vert\right)\\
&+ \left(q^{-}(Y^{i}, M^{N,2}) - q^{-}(X^{i}, M^{N,1})\right)_{+}\left( \vert X^{i} - Y^{i} +1 \vert - \vert X^{i} - Y^{i} \vert\right).\\
\end{align*}

Using the fact that for all $x,y\geq 0$, $(x-y)_{+} + (y-x)_{+} = \vert x-y \vert$, we have,
\begin{align*}
H_i
&= \left((q^{+}-q^{-})(X^{i}, M^{N,1}) - (q^{+}-q^{-})(Y^{i}, M^{N,2})\right)\mathds{1}_{X^{i} > Y^{i}}\\
&+ \left((q^{+}-q^{-})(Y^{i}, M^{N,2}) - (q^{+}-q^{-})(X^{i}, M^{N,1})\right)\mathds{1}_{X^{i} < Y^{i}}\\
&+ \Big( \vert q^{+}(X^{i}, M^{N,1}) - q^{+}(Y^{i}, M^{N,2}) \vert + \vert q^{-}(X^{i}, M^{N,1}) - q^{-}(Y^{i}, M^{N,2})\vert \Big) \mathds{1}_{X^{i} = Y^{i}}.
\end{align*}

Under Assumption $(B)$, the growth of $q^{+}$ and the decrease of $q^{-}$ on the second component imply, for $i\in \{1, \dots, N\}$ such that $X^{i}=Y^{i}$
\begin{align*}
J_i 
&:= \vert q^{+}(X^{i}, M^{N,1}) - q^{+}(X^{i}, M^{N,2}) \vert + \vert q^{-}(X^{i}, M^{N,1}) - q^{-}(X^{i}, M^{N,2})\vert\\
&= \left((q^{+}-q^{-})(X^{i}, M^{N,2}) - (q^{+}-q^{-})(X^{i}, M^{N,1})\right)\mathds{1}_{X^{i} = Y^{i}} \mathds{1}_{M^{N,1} < M^{N,2}}\\
&+ \left((q^{+}-q^{-})(X^{i}, M^{N,1}) - (q^{+}-q^{-})(X^{i}, M^{N,2})\right)\mathds{1}_{X^{i} = Y^{i}} \mathds{1}_{M^{N,1} > M^{N,2}}.
\end{align*} 

We deduce that, under Assumption (B), there exists $\alpha>0$ such that for $i \in \{1, \dots, N\}$
\begin{align*}
H_i
&\leq \alpha \left(\vert X^{i}-Y^{i}\vert + \vert M^{N,1}-M^{N,2}\vert \right),
\end{align*}

and thus the definition of $M^{N,l}, l=1,2$ implies that
\begin{align*}
\mathbb{L}_2 d(X,Y) 
&\leq 2\alpha d(X,Y). 
\end{align*}

\bigskip

We deduce that $\mathbb{L} d(X,Y) \leq -(\lambda-2\alpha) d(X,Y).$ Now let $(\mathbb{P}_{t})_{t \geq 0}$ be the semi-group associated with the generator $\mathbb{L}$. Using the equality $\partial_{t} \mathbb{P}_{t}f = \mathbb{P}_{t}\mathbb{L}f$ and Gronwall's Lemma, we have, for every $t\geq 0$, $\mathbb{P}_t d \leq e^{- (\lambda-2\alpha) t} d$; namely
$$
\E[d(X_t,Y_t)]
\leq e^{-(\lambda-2\alpha) t} \E[d(X_0, Y_0)].
$$
Taking the infimum over all couples $(X_{0},Y_{0})$, the claim follows.
\end{proof}





\bigskip

\textbf{A new condition on the interaction rates.}
If we replace the Lipschitz condition in Assumption $(B)$ by the following one:
there exist $\alpha, \zeta >0$ such that for any $(k_1,l_1), (k_2, l_2) \in \mathbb{N}\times \mathbb{R}_+$, $k_1 \geq k_2$
\begin{equation}
\label{eq:convexCondition}
(q^{+}-q^{-})(k_1,l_1)-(q^{+}-q^{-})(k_2,l_2) \leq -\alpha \left(k_1-k_2\right) + \zeta \left(l_1-l_2\right).
\end{equation} 

Then, under this condition and Assumption $(A)$, we have for any processes $(X_t^{N})_{t\geq 0}$ and $(Y_t^{N})_{t\geq 0}$ generated by \eqref{eq:generator} and for any $t\geq 0$,
\begin{equation*}
\mathcal{W}(\mathrm{Law}(X_t^{N}), \mathrm{Law}(Y_t^{N})) \leq e^{-\left((\lambda+\alpha)- \zeta\right) t} \mathcal{W}(\mathrm{Law}(X_0^{N}), \mathrm{Law}(Y_0^{N})).
\end{equation*} 

If $l_1=l_2$, the condition \eqref{eq:convexCondition} is a convexity condition on the first variable. For $l_1 \neq l_2$, the term $\zeta (l_1-l_2)$ represents the fluctuations of the barycenters. The resulting rate is slightly better than the one before, but we can find an example of interaction rates for which we obtain an optimal rate of convergence in Theorem \ref{th:coupling}. This example is inspired by Malrieu's model when the interaction potential is $W(x,y)=a(x-y)^2$, $a>0$. 
We assume that for a particle system $X=(X^{1}, \dots, X^{N})$, the interaction birth and death rates are given respectively by
$$
q_X^{+}(X^{i}) = a\dfrac{1}{N} \sum_{j=1}^{N} \left(X^{i} - X^{j}\right)_{-} \ \text{and} \ q_X^{-}(X^{i}) = a\dfrac{1}{N} \sum_{j=1}^{N} \left(X^{i} - X^{j}\right)_{+},
$$
where $a>0$, and the generator is given by
\begin{align}
\label{eq:generatorOptimal}
\mathcal{L}f(x) 
&= \sum_{i=1}^{N} \Big[\left(b_{x_{i}} + q_{\cdot}^{+}(x_{i})\right) \left(f(x+e_i)-f(x)\right) +\left(d_{x_{i}} + q_{\cdot}^{-}(x_{i})\right) \left(f(x-e_i)-f(x)\right) \mathds{1}_{x_i>0} \Big].
\end{align}

\begin{theo} Assume that Assumption $(A)$ is satisfied. Then, for the processes $(X_t^{N})_{t\geq 0}$ and $(Y_t^{N})_{t\geq 0}$ generated by \eqref{eq:generatorOptimal}, we have for all $t\geq 0$
\begin{equation*}
\mathcal{W}(\mathrm{Law}(X_t^{N}), \mathrm{Law}(Y_t^{N})) \leq e^{-\lambda t} \mathcal{W}(\mathrm{Law}(X_0^{N}), \mathrm{Law}(Y_0^{N})).
\end{equation*} 
\end{theo}

\begin{proof}
Using the same coupling and the same notations as in the proof of Theorem \ref{th:coupling}, we have
$$
\mathbb{L}_{2} d(X,Y) = \sum_{i=1}^{N} H_i,
$$
where for all $i \in \{1, \dots, N\}$
\begin{align*}
H_i
&= \left((q_X^{+}-q_X^{-})(X^{i}) - (q_Y^{+}-q_Y^{-})(Y^{i})\right)\mathds{1}_{X^{i} > Y^{i}}\\
&+ \left((q_Y^{+}-q_Y^{-})(Y^{i}) - (q_X^{+}-q_X^{-})(X^{i}, M^{N,1})\right)\mathds{1}_{X^{i} < Y^{i}}\\
&+ \Big( \vert q_X^{+}(X^{i}) - q_Y^{+}(X^{i}) \vert + \vert q_X^{-}(X^{i}) - q_Y^{-}(X^{i})\vert \Big) \mathds{1}_{X^{i} = Y^{i}}.
&\quad \quad \\
&\leq a\left[-\left(X^{i}-Y^{i}\right) + \dfrac{1}{N} \sum_{i=1}^{N} \left(X^{j}-Y^{j}\right)\right]\mathds{1}_{X^{i} > Y^{i}}\\
&+ a\left[-\left(Y^{i}-X^{i}\right) + \dfrac{1}{N} \sum_{i=1}^{N} \left(Y^{j}-X^{j}\right)\right]\mathds{1}_{Y^{i} > X^{i}}\\
&+ a\dfrac{1}{N} \sum_{i=1}^{N} \vert X^{j}-Y^{j}\vert \mathds{1}_{X^{i} = Y^{i}}.
\end{align*}

Thus,
$$
H_i \leq - a\vert X^{i}-Y^{i}\vert + a\dfrac{1}{N} \sum_{i=1}^{N} \vert X^{j}-Y^{j}\vert,
$$
and
$$
\sum_{i=1}^{N} H_i \leq 0.
$$

We deduce that $\mathbb{L}d(X,Y) \leq -\lambda d(X,Y).$
\end{proof}

\bigskip

\section{Proof of Theorem \ref{th:PropChaos}}
\label{sect:proofProgChaos}
Let us give an important consequence of Theorem \ref{th:PropChaos}: with explicit rate, we have the propagation of chaos for the system of interacting particles.
\begin{coro}[Strong Propagation of chaos]
Let $\mu_t^{(k,N)}$ be the law of $k$ particles among $N$ at time $t$ and $u_t$ be the law of the nonlinear process. Then,
$$
\sup_{t\geq 0} \mathcal{W}(\mu_t^{(k,N)}, u_t^{\otimes k}) \leq \dfrac{kK}{\sqrt{N}}.
$$
\end{coro}

\begin{proof}[Proof]
By exchangeability, the $k$-marginals of $\mathrm{Law}(X_t^{N})$ do not depend on the choice of coordinates. Thus,
$$
\mathcal{W}(\mu_t^{(k,N)}, u_t^{\otimes k}) \leq \sum_{i=1}^{k} \E[\vert X_t^{i,N} - \overline{X}_t^{i} \vert].
$$
\end{proof}

\begin{proof}[Proof of Theorem \ref{th:PropChaos}]
To prove the propagation of chaos phenomenon, we construct a coupling between the particle system $(X^{1,N}, \dots, X^{N,N})$ and $N$ independent nonlinear processes $(\overline{X}^{i}, \dots, \overline{X}^{N})$. For $i \in \{1, \dots, N\}$ and $t\geq 0$
\begin{itemize}
\item $\overline{X}_0^{i} = X_0^{i,N}$
\item $\mathrm{Law}(\overline{X}_{t}^{i}) =u_t$
\item the transition rates of $\overline{X}^{i}$ at time $t$ are given by
\[ 
  \begin{array}{ c c c c c}
     i & \rightarrow & i+1 & \text{with rate} \ b_i + q^{+}(i,\Vert u_t\Vert), &\\
     i & \rightarrow & i-1 & \text{with rate} \ d_i + q^{-}(i,\Vert u_t\Vert), & \text{for} \ i\geq 1\\
     i & \rightarrow & j & \text{with rate} \ 0, & \ \text{if} \ j\notin \{i-1,i+1\}.\\
  \end{array} 
\]
\end{itemize}

Using the coupling and the notations introduced in the proof of Theorem \ref{th:coupling},   we have under Assumption $(A)$
\begin{align*}
\mathbb{L}_{1} d(X,\overline{X}) &\leq -\lambda d(X,\overline{X}),
\end{align*}
and under Assumption $(B)$
\begin{align*}
\mathbb{L}_{2}d(X, \overline{X})
&\leq \alpha d(X, \overline{X}) + \alpha N \vert M^{N} - \Vert u \Vert \vert,
\end{align*}
where we recall that the distance $d$ is the $l^{1}$-distance.
But, for $t\geq 0$
\begin{align*}
\E \vert M_{t}^{N} - \Vert u_{t} \Vert \vert
&\leq \E \Big\vert M_{t}^{N} - \dfrac{1}{N} \sum_{i=1}^{N} \overline{X}_{t}^{i} \Big\vert + \E \Big\vert \dfrac{1}{N} \sum_{i=1}^{N} \overline{X}_{t}^{i} - \Vert u_t \Vert \Big\vert.
\end{align*} 

On the one hand
\begin{align*}
\E \Big\vert M_{t}^{N} - \dfrac{1}{N} \sum_{i=1}^{N} \overline{X}_{t}^{i} \Big\vert
&\leq \dfrac{1}{N} \E \left( \sum_{i=1}^{N} \vert X_{t}^{i,N} - \overline{X}_{t}^{i}\vert \right) = \dfrac{1}{N}\E(d(X_t, \overline{X}_t)).
\end{align*}
On the other hand, by independence and Cauchy Schwarz inequality 
\begin{align*}
\E \Big\vert \dfrac{1}{N} \sum_{i=1}^{N} \overline{X}_{t}^{i} - \Vert u_t \Vert \Big\vert
&= \E \Big\vert \dfrac{1}{N} \sum_{i=1}^{N} \overline{X}_{t}^{i} - \E(\overline{X}_t^{1}) \Big\vert\\
&= \dfrac{1}{N} \E \Big\vert \sum_{i=1}^{N} \overline{X}_{t}^{i} - \E \left(\sum_{i=1}^{N} \overline{X}_{t}^{i}\right) \Big\vert\\
&\leq \dfrac{1}{N} \left( \var\left( \sum_{i=1}^{N} \overline{X}_{t}^{i}\right) \right)^{\frac{1}{2}}\\
&\leq \dfrac{1}{N} \left(\sum_{i=1}^{N} \var( \overline{X}_{t}^{i})\right)^{\frac{1}{2}}\\
& \leq \dfrac{1}{\sqrt{N}} \var(\overline{X}_{t}^{1})^{\frac{1}{2}}.
\end{align*}

Now, the moments of the process $(\overline{X}_t)_{t\geq 0}$ are bounded. More precisely, the process $(\overline{X}_t)_{t\geq 0}$ has finite exponential moments, uniform in time, as soon as it is finite at time $0$.

\begin{lem}[Exponential moment of $(\overline{X}_t)_{t\geq 0}$]
\label{lem:expoMoment} Let $\delta>0$ and $\beta(\delta):= \displaystyle \inf_{x\in \mathbb{N}^{*}}\left(d_x e^{-\delta} -b_x\right)$ and let $K_1 = \alpha \left( \Vert u_0\Vert + \dfrac{b_0}{\lambda-2\alpha} \right)$. Then, under Assumptions $(A)$ and $(B)$ and if $\beta(\delta)-K_1>0$, $\displaystyle \sum_{i\geq 0} e^{\delta i} u_t(i)$ is finite for every $t\geq 0$ as soon as $\displaystyle \sum_{i\geq 0} e^{\delta i} u_0(i)$ is finite and
$$
\sum_{i\geq 0} e^{\delta i} u_t(i) \leq \sum_{i\geq 0} e^{\delta i} u_0(i) + \dfrac{b_0}{\beta(\delta)-K_1}.
$$
\end{lem}

\begin{proof}[Proof of Lemma \ref{lem:expoMoment}]
Let us first remark that, if $\lambda-2\alpha>0$ 
$$
\Vert u_t\Vert \leq \Vert u_0\Vert + \dfrac{b_0}{\lambda-2\alpha}.
$$
Indeed, applying the operator $\mathcal{G}_{(\cdot)}$ defined in \eqref{eq:generatorParticle} to the function $f(i)=i$ we have
\begin{align*}
\mathcal{G}_{u_t}f(i)
&= b_i - d_i + \left( q^{+} -q^{-}\right)\left(i, \Vert u_t\Vert \right)\\
&\leq -\lambda i + b_0 + \alpha\left( i+ \Vert u_t\Vert\right).
\end{align*}
By equation \eqref{eq:nonlinearMasterEq} we obtain
$$
\dfrac{d}{dt}\Vert u_t\Vert \leq -\left(\lambda-2\alpha\right) \Vert u_t \Vert + b_0
$$
and Gronwall's lemma gives the result.\\
Now, let us take $f(i) = e^{\delta i}$, $\delta>0$. Then, under Assumption $(B)$
\begin{align*}
\mathcal{G}_{u_t}f(i)
&\leq -(e^{\delta}-1) e^{\delta i} \left[ e^{- \delta}d_{i} - b_{_i}\right]\mathds{1}_{i >0} + (e^{\delta}-1) e^{\delta i} q^{+}(i, \Vert u_t\Vert) + b_0 (e^{\delta}-1)\\
&\leq - \beta(\delta) (e^{\delta}-1)f(i) + (e^{\delta}-1)f(i) q^{+}(0, \Vert u_t\Vert) + b_0 (e^{\delta}-1)\\
&\leq - \beta(\delta) (e^{\delta}-1)f(i) + \alpha(e^{\delta}-1)f(i) \Vert u_t\Vert + b_0 (e^{\delta}-1)\\
&\leq - \left(\beta(\delta)-K_1\right) (e^{\delta}-1)f(i) + b_0 (e^{\delta}-1).
\end{align*}
Thus,
$$
\dfrac{d}{dt}\sum_{i\geq 0} e^{\delta i} u_t(i) \leq - \left(\beta(\delta) -K_1\right) (e^{\delta}-1) \sum_{i\geq 0} e^{\delta i} u_t(i) + b_0 (e^{\delta}-1).
$$
\end{proof}

We are now able to conclude the proof. For $t\geq 0$ and $i\in \{1, \dots, N\}$, let
$$
\gamma(t) = \E \vert X_t^{i,N} - \overline{X}_t^{i} \vert.
$$

Then, by the equality $\partial_{t} \mathbb{P}_{t}f = \mathbb{P}_{t}\mathbb{L}f$ (where we recall that $\mathbb{P}$ is the semi-group associated with $\mathbb{L}$), Lemma \ref{lem:expoMoment} and the exchangeability of the marginals of the particle system, there exists $K>0$ such that 
$$
\partial_t \gamma(t) \leq -(\lambda-2\alpha) \gamma(t) + \dfrac{K}{\sqrt{N}}.
$$
Gronwall's lemma gives for every $t\geq 0$
$$
\gamma(t) \leq e^{-(\lambda-2\alpha) t} \gamma(0) + \dfrac{K}{\sqrt{N}} \left( 1- e^{-(\lambda-2\alpha)t} \right).
$$
As the initial conditions are the same, we obtain \eqref{eq:chaos}.
\end{proof}

\bigskip

\begin{proof}[Proof of Corollary \ref{cor:ConvEmpirique}]
Let $\varphi$ be a function such that $\Vert \varphi \Vert_{Lip} \leq 1$. Then, 
using the same coupling as Theorem \ref{th:PropChaos}, we have
\begin{align*}
\E \vert \mu_t^{N}(\varphi) -u_t(\varphi) \vert
&= \E \Big\vert \dfrac{1}{N}\sum_{i=1}^{N} \varphi(X_t^{i,N}) - \E[\varphi(\overline{X}_{t}^{1})] \Big\vert\\
&\leq \E \Big\vert \dfrac{1}{N}\sum_{i=1}^{N} \varphi(X_t^{i,N}) - \dfrac{1}{N}\sum_{i=1}^{N} \varphi(\overline{X}_t^{i}) \Big\vert + \E \Big\vert \dfrac{1}{N}\sum_{i=1}^{N} \varphi(\overline{X}_t^{i}) - \E[\varphi(\overline{X}_{t}^{1})] \Big\vert.
\end{align*}
By Theorem \ref{th:PropChaos}, there exists $K>0$ such that
$$
\E \Big\vert \dfrac{1}{N}\sum_{i=1}^{N} \varphi(X_t^{i,N}) - \dfrac{1}{N}\sum_{i=1}^{N} \varphi(\overline{X}_t^{i}) \Big\vert \leq \dfrac{K}{\sqrt{N}}.
$$
Now, the Cauchy-Schwarz inequality and Lemma \ref{lem:expoMoment} imply the existence of a constant $C>0$ such that
$$
 \E \Big\vert \dfrac{1}{N}\sum_{i=1}^{N} \varphi(\overline{X}_t^{i}) - \E[\varphi(\overline{X}_{t}^{1})] \Big\vert \leq \dfrac{C}{\sqrt{N}}.
$$ 
\end{proof}

\begin{proof}[Proof of Corollary \ref{cor:devInequality}]
Let $\nu^{N}$ be the empirical measure of the independent processes $(\overline{X}^{i})_{i\in \{1, \dots, N\}}$. Namely, for any $t\geq 0$
$$
\nu_t^{N} = \dfrac{1}{N} \sum_{i=1}^{N} \delta_{\overline{X}_{t}^{i,N}}.
$$
Then, the triangular inequality gives
$$
\mathcal{W}(\mu_t^{N}, u_t) \leq \mathcal{W}(\mu_t^{N}, \nu_t^{N}) + \mathcal{W}(\nu_t^{N}, u_t).
$$

As for any $t\geq 0$,
$$
\mathcal{W}(\mu_t^{N}, \nu_t^{N}) \leq \dfrac{1}{N} \sum_{i=1}^{N} \vert X_t^{i,N} - \overline{X_t}^{i} \vert,
$$

then, by Theorem \ref{th:PropChaos}, there exists a constant $K>0$ such that
\begin{align*}
\mathbb{E}[\mathcal{W}(\mu_t^{N}, \nu_t^{N})]
&\leq \dfrac{1}{N} \sum_{i=1}^{N} \sup_{t\geq 0} \mathbb{E} \vert X_t^{i,N} - \overline{X_t}^{i} \vert\\
&\leq \dfrac{K}{\sqrt{N}}.
\end{align*}

On the other hand, by Lemma \ref{lem:expoMoment}, the process $\overline{X}^{i}$, for every $i\in \{1, \dots, N\}$, has finite exponential moments. So, applying Theorem 1 of \cite{FG15} with $p=d=1$ and $q>2$, there exists a constant $\overline{K}$ such that
\begin{align*}
\mathbb{E}[\mathcal{W}(\nu_t^{N}, u_t)]
&\leq \dfrac{\overline{K}}{\sqrt{N}}.
\end{align*}

We deduce that for every $t\geq 0$,
$$
\mathbb{E}[\mathcal{W}(\mu_t^{N}, u_t)] \leq \dfrac{K+\overline{K}}{\sqrt{N}},
$$
which by Markov's inequality ends the proof.
\end{proof}

\section{Proof of Theorem \ref{th:LTBnonlinear}}
\label{sect:proofNonlinear}

\begin{proof}[Proof of Theorem \ref{th:LTBnonlinear}]
Let $(X_t^{N})_{t \geq 0}$ and $(Y_t^{N})_{t \geq 0}$ be two particle systems generated by \eqref{eq:generator} with initial laws $u_0^{\otimes N}$ and $v_0^{\otimes N}$ respectively. Let $u_t^{1,N}$ (resp. $v_t^{1,N}$) be the first marginal of $\mathrm{Law}(X_t^{N})$ (resp. $\mathrm{Law}(Y_t^{N})$). 
Then, the triangular inequality yields
$$
\mathcal{W}(u_t, v_t) \leq \mathcal{W}(u_t, u_t^{1,N}) + \mathcal{W}(u_t^{1,N}, v_t^{1,N}) + \mathcal{W}(v_t^{1,N},v_t).
$$
The uniform propagation of chaos (Theorem \ref{th:PropChaos}) gives for every $t\geq 0$
\begin{equation*}
\mathcal{W}(u_t, u_t^{1,N}) \leq \E \vert \overline{X}_t^{1} - X_t^{1,N} \vert \leq \dfrac{K}{\sqrt{N}},
\end{equation*}
and
\begin{equation*}
\mathcal{W}(v_t, v_t^{1,N}) \leq \dfrac{K}{\sqrt{N}}.
\end{equation*}
Now, by exchangeability of the marginals of the particles and by Theorem \ref{th:coupling} we have
\begin{align*}
\mathcal{W}(u_t^{1,N}, v_t^{1,N}) 
&\leq \dfrac{1}{N} \mathcal{W}(\mathrm{Law}(X_t^{N}), \mathrm{Law}(Y_t^{N}))\\
&\leq \dfrac{1}{N} e^{-(\lambda-2\alpha) t} \mathcal{W}(\mathrm{Law}(X_0^{N}), \mathrm{Law}(Y_0^{N}))\\
&\leq \dfrac{1}{N} e^{-(\lambda-2\alpha) t} N \mathcal{W}(u_0, v_0)\\
&\leq e^{-(\lambda-2\alpha) t} \mathcal{W}(u_0, v_0).
\end{align*}
We deduce that 
$$
\mathcal{W}(u_t, v_t) \leq \dfrac{2K}{\sqrt{N}} + e^{-(\lambda-2\alpha) t} \mathcal{W}(u_0, v_0).
$$
Taking the limit as $N$ tends to infinity we obtain \eqref{eq:CauchyW1}. Then, the sequence $(u_t)_{t\geq 0}$ is a Cauchy sequence for the $\mathcal{W}_1$-Wasserstein distance and thus admits a limit $u_{\infty}$.
\end{proof}

\section{Appendix}
\label{sect:annexe}
In the following theorem, we prove the existence of a Lyapunov function.
This function ensures the non-explosion and the positive recurrence of the particle system.
Along this section, we note $\kappa= \lambda-2\alpha$.
\begin{theo}[Lyapunov function]
\label{th:lyapunov}
Let $V$ be the function $x\in \mathbb{N}^{N} \mapsto \displaystyle \sum_{k=1}^{N} x_k$ and let us assume that Assumptions $(A)$ and $(B)$ are satisfied.
Then, for $x\in \mathbb{N}^{N}$
\begin{equation}
\label{eq:lya}
\mathcal{L}V(x) \leq -\kappa V(x) +  b_0 N.
\end{equation}
Thus, the function $V$ is a Lyapunov function for $\mathcal{L}$. In particular, the process $(X_t^{N})_{t\geq 0}$ is
non-explosive and positive recurrent.
\end{theo}

\bigskip

Under this assumption and without rate of convergence, the existence of a Lyapunov function combined with the irreducibility of the process (which implies that every compact is a small set) provides a sufficient criterion ensuring that the interacting particle system is ergodic \cite[Theorem 6.1]{MT93}. And the inequality below implies that under the invariant distribution, the mean of the particle system $M^{N}$ is upper bounded.

\begin{proof}[Proof]
\textbf{Lyapunov function.}\\
Let $V$ be the function $x\mapsto \displaystyle \sum_{k=1}^{N} x_k$. Then,
\begin{align*}
\mathcal{L}V(x) 
&= \sum_{i=1}^{N} \left[ b_{x_i} +q^{+}(x_i,M^{N}) - \left(d_{x_i} + q^{-}(x_i,M^{N})\right)\mathds{1}_{x_i >0}\right]\\
&= \sum_{i=1}^{N} \left[ \left( b_{x_i} - d_{x_i}\right) + (q^{+}- q^{-})(x_i,M^{N})\right]\mathds{1}_{x_i >0} + \left(b_0+ q^{+}(0,M^{N}) \right) \sum_{i=1}^{N} \mathds{1}_{x_i=0}.
\end{align*}

By Assumption $(A)$
\begin{align*}
d_{x_i} - b_{x_i}
&= \sum_{n=0}^{x_i -1} \left( d_{n+1} -d_n -b_{n+1} +b_n \right) -b_0\\
&\geq \lambda x_i -b_0,
\end{align*}

and Assumption $(B)$ gives
\begin{align*}
(q^{+}- q^{-})(x_i,M^{N}) \leq \alpha \left( x_i + M^{N}\right) \ \text{and} \ q^{+}(0,M^{N}) \leq \alpha M^{N} .
\end{align*}

Thus,
\begin{align*}
\mathcal{L}V(x)
&\leq -\lambda \sum_{i=1}^{N} x_i + b_0 \sum_{i=1}^{N} \mathds{1}_{x_i\geq 0} +\alpha \sum_{i=1}^{N} x_i + \alpha M^{N} \sum_{i=1}^{N} \mathds{1}_{x_i\geq 0}\\
&\leq -\kappa V(x) +  b_0 N.
\end{align*}

\bigskip 

\textbf{Non-explosion.}\\
Let $V$ be the Lyapunov function defined in Theorem \ref{th:lyapunov}, and for $x\in \mathbb{N}^{N}$ let $W$ be the function defined by
$$
W(x)= \dfrac{1}{N} \left( V(x)- \dfrac{b_0 N}{\kappa} \right).
$$

Then $W$ satisfies a simpler inequality than \eqref{eq:lya} given by
$$
\mathcal{L}W(x) \leq -\kappa W(x).
$$

Consider now the function $f:\mathbb{N}^{N} \times \mathbb{R}_+ \rightarrow \mathbb{R}$ defined by $f(x,t) = e^{\kappa t} W(x)$ and for $A$, let $\tau_A$ be the stopping time defined by $\tau_A= \inf \{t\geq 0, W(X_t)>A \}$. Then, as
$$
M_t = f(X_t,t)-f(X_0,0) - \int_{0}^{t} \left(\dfrac{\partial}{\partial_t}f + \mathcal{L}f\right)(X_s,s) ds \quad \text{is a martingale,}
$$
we have for $x\in \mathbb{N}^{N}$
$$
\E_x [f(X_{t\wedge \tau_A},t\wedge \tau_A)]= W(x)+\E_x \left[\int_{0}^{t\wedge \tau_A} \left(\dfrac{\partial}{\partial_t}f + \mathcal{L}f\right)(X_s,s) ds \right]. 
$$
But,
$$
\dfrac{\partial}{\partial_t}f + \mathcal{L}f \leq 0,
$$
thus,
\begin{align*}
\E_x[e^{\kappa (t\wedge \tau_A)}W(X_{t\wedge \tau_A})]
&\leq W(x).
\end{align*}
And finally by the definition of $\tau_A$ and $W$
\begin{align*}
\mathbb{P}_x (\tau_A <t) 
&\leq \dfrac{1}{A}\E_x \Big[ e^{\kappa \tau_A}W(X_{\tau_A}) \mathds{1}_{\tau_A \leq t}\Big]\\
&\leq \dfrac{1}{A} \left( W(x) - e^{\kappa t} \E_x[W(X_t) \mathds{1}_{\tau_A \geq t}] \right)\\
&\leq \dfrac{1}{A} \left(W(x)+ b_0 e^{\kappa t} \right).
\end{align*}

\bigskip 

\textbf{Positive recurrence.}\\
Let $A=\{x\in \mathbb{N}^{N}: \kappa V(x) < 2b_0 N \}$, $H_{A}$ the hitting time of $A$ and $H_{A}^{n} =H_{A}\wedge n$. As $\lim \limits_{\Vert x \Vert\rightarrow +\infty} V(x)=+\infty$, then $A$ is a finite set.
Applying the same argument for $f(x,t)= e^{\frac{\kappa}{2}t} V(x)$ gives
\begin{align*}
\E_x[e^{\frac{\kappa}{2}H_A^{n}}V(X_{H_A^{n}})]
&\leq V(x) \quad \forall x\notin A.
\end{align*}

By definition of $H_A$, we deduce that
$$
\E_x[e^{\frac{\kappa}{2} H_{A}^{n}}] \leq \dfrac{V(x)\kappa}{2b_0 N} \quad \forall x\notin A,
$$
and the monotonicity convergence theorem gives
\begin{equation}
\label{eq:posRecurrence}
\E_x[e^{\frac{\kappa}{2} H_{A}}] \leq \dfrac{V(x)\kappa}{2b_0N} \quad \forall x\notin A.
\end{equation}

Now, let $\sigma_A = \inf \{t > H_{A^{c}}, X_t \in A\}$ and $x\in A$. Using \eqref{eq:posRecurrence}, we have
\begin{align*}
\E_x[e^{\frac{\kappa}{2} \sigma_A}]
&= \E_x[\E_{X_{H_{A^{c}}}}[e^{\frac{\kappa}{2} \sigma_A}]]
\leq \dfrac{\kappa \E_x[V(X_{H_{A^{c}}})]}{2b_0 N} 
\end{align*} 

Using the fact that $\mathcal{L}V(x) < +\infty$ for $x\in A$ ends the proof.
\end{proof}

\bigskip

\textbf{Acknowledgement:} I would like to thank Philippe Robert for the paper's reference \cite{DTZ05} and my thesis advisors Amine Asselah, Djalil Chafa\"{i} for valuable discussions. I would also like to thank Bertrand Cloez, Florent Malrieu and Pierre-Andr\'{e} Zitt for valuable discussions.
The work was supported by grants from R\'{e}gion Ile-de-France and I thank the support of the A*MIDEX grant
(n$°$ANR-11-IDEX-0001-02) funded by the French Government "Investissements d'Avenir" program.

\bibliographystyle{abbrv} 

\end{document}